\documentclass[12pt]{article}
\usepackage[]{amsthm}
\usepackage[]{amsmath}
\usepackage[]{hyperref}

\newtheorem{proposition}{Proposition}
\newtheorem{corollary}{Corollary}

\title{The Meta-C-Finite Ansatz}
\author{Robert Dougherty-Bliss}
\date{\today}

\begin{document}

\maketitle

\begin{abstract}
    \noindent The Fibonacci numbers satisfy the famous recurrence $F_n = F_{n -
    1} + F_{n - 2}$. The theory of C-finite sequences ensures that the
    Fibonacci numbers whose indices are divisible by $m$, namely $F_{mn}$,
    satisfy a similar recurrence for every positive integer $m$, and these
    recurrences have an explicit, uniform representation. We will show that
    $a(mn)$ has a uniform recurrence over $m$ for any C-finite sequence $a(n)$
    and use this to automatically derive some famous summation identities.
\end{abstract}

\noindent The Fibonacci numbers $F_n$ satisfy the famous recurrence $F_n = F_{n
- 1} + F_{n - 2}$. The sequence which takes every \emph{other} Fibonacci
number, $F_{2n},$ satisfies the similar recurrence $F_{2n} = 3 F_{2(n - 2)} -
F_{2(n - 2)}$. In fact, every sequence of the form $F_{mn}$ satisfies such a
recurrence. Here are the first few:
\begin{align}
    \label{fib-rec}
    \begin{split}
        F_{n} &= F_{n - 1} + F_{n - 2} \\
        F_{2n} &= 3 F_{2(n - 1)} - F_{2(n - 2)} \\
        F_{3n} &= 4 F_{3(n - 1)} + F_{3(n - 2)} \\
        F_{4n} &= 7 F_{4(n - 1)} - F_{4(n - 2)} \\
        F_{5n} &= 11 F_{5(n - 1)} + F_{5(n - 2)}.
    \end{split}
\end{align}
If we look closely at the coefficients that appear---or plug them into the OEIS
\cite{oeis}---there seems to be a general recurrence:
\begin{equation}
    \label{fib-full-rec}
    F_{mn} = L_m F_{m(n - 1)} + (-1)^{m + 1} F_{m(n - 2)}.
\end{equation}
This conjecture is right on the money, and we can prove it a dozen different
ways---Binet's formula, induction, generatingfunctionology---but the
\emph{outline} is more interesting.

We began with a sequence which satisfied a nice recurrence ($F_n$), examined
recurrences for a family of related sequences $(F_{mn}$), then noticed that the
coefficients on the recurrences satisfied a \emph{meta pattern} (equation
\eqref{fib-full-rec}). This outline holds for any sequence which satisfies a
linear recurrence relation with constant coefficients. Such sequences are
called \emph{C-finite} \cite{ansatz, concrete}.

The remainder of the paper is organized as follows. 
Section~\ref{sec:definitions} gives a brief overview of C-finite sequences,
Section~\ref{sec:uniform_recurrences} proves that an analogue of
\eqref{fib-full-rec} holds for any C-finite sequence,
Section~\ref{sec:uniform_products} shows that a similar property holds for
products of C-finite sequences, and Section~\ref{sec:uniform_sums} applies some
of our results to produce infinite families of summation identities.

\section{The C-finite ansatz}%
\label{sec:definitions}

The theory of C-finite sequences is beautifully laid out in \cite{concrete} and
\cite{ansatz}. What follows is a brief description of the principle results.
For simplicity, assume that everything we do is over an algebraically closed
field such as the complex numbers.

Given a sequence $a(n)$, let $N$ be the shift operator defined by
\begin{equation*}
    Na(n) = a(n + 1).
\end{equation*}
We say that $a(n)$ is \emph{C-finite} if and only if there exists a polynomial
$p(x)$ such that $p(N) a(n) = 0$ for all $n \geq 0$. We say that $p(x)$
\emph{annihilates} $a(n)$. For example, $x^2 - x - 1$ annihilates the Fibonacci
sequence $F(n)$ and $x - 2$ annihilates the exponential sequence $2^n$. The set
of all polynomials which annihilate a fixed $a(n)$ is an ideal. The generator
of this ideal is the \emph{characteristic polynomial} of $a(n)$, and we call
its degree the degree (or order) of $a(n)$.

Every C-finite sequence has a closed-form expression as a sum of polynomials
times exponential sequences. More specifically,
\begin{equation*}
    a(n) = \sum_{k = 1}^m f_k(n) r_k^n,
\end{equation*}
where $r_1, r_2, \dots, r_m$ are the distinct roots of the characteristic
equation of $a(n)$ and $f_k(n)$ is a polynomial in $n$ with degree less than or
equal to the multiplicity of the root $r_k$. We call these formulas
\emph{Binet-type formulas} after Binet's famous formula for the Fibonacci
numbers. For example, $(x - 2)^2$ is an annihilating polynomial of any sequence
$a(n)$ which satisfies the recurrence $a(n + 2) = 4 a(n + 1) - 4 a(n)$, and
this implies $a(n) = (\alpha + \beta n)2^n$ for some constants $\alpha$ and
$\beta$.

We can go the other way and derive an annihilating polynomial from a closed
form expression. A term of the form $n^d r^n$ is annihilated by $(x - r)^{d +
1}$, so for each exponential $r^n$ in the closed form, look for the highest
power $n^d$ which is multiplied by $r^n$ and write down $(x - r)^{d + 1}$. For
example, the sequence $a(n) = n 3^n - \frac{n^2}{2} + 5^n$ is annihilated by
$(x - 3)^2 (x - 1)^3 (x - 5)$.

Finally, if $a(n)$ and $b(n)$ are two C-finite sequences, then so are the
following:
\begin{equation*}
    a(n) b(n) \qquad a(n) \pm b(n) \qquad \sum_{k = 0}^n a(k) b(n - k).
\end{equation*}

C-finite sequences are a special subclass of \emph{holonomic sequences},
sequences which satisfy a linear recurrence with \emph{polynomial coefficients}
\cite{holonomic}. Holonomic sequences satisfy very similar properties, but do
not have the readily computable closed forms which we need here.

\section{Uniform recurrences}%
\label{sec:uniform_recurrences}

First up, we will prove the analogue of \eqref{fib-full-rec} for arbitrary
C-finite sequences.

\begin{proposition}
    \label{uniform-rec}
    If $a(n)$ is a C-finite sequence of order $d$, then $n \mapsto a(nm)$
    satisfies a recurrence of the form
    \begin{equation}
        a(nm) = \sum_{k = 1}^d c_k(m) a((n - k)m),
    \end{equation}
    where $c_k(m)$ is C-finite with respect to $m$ and has order at most ${d
    \choose k}$. The sequence $c_1(m)$ always satisfies the same recurrence as
    $a(n)$ itself, and $c_d(k) = \omega^k$, where $\omega$ is $(-1)^d$ times
    the constant coefficient of the characteristic polynomial of $a(n)$.
\end{proposition}

The following proof is constructive given the roots of the characteristic
polynomial of $a(n)$, but \cite{birmajer} gives formulas for $c_k(m)$ in terms
of partial Bell polynomials without reference to the roots.

\begin{proof}
    The Binet-type formula for $a(n)$ is a linear combination of terms of the
    form $n^i r^n$ where $i$ is a nonnegative integer and $r$ is a root of the
    characteristic polynomial of $a(n)$. Thus, the Binet-type formula for
    $a(nm)$ is a linear combination of terms of the form $(nm)^i r^{nm}$, which
    is equivalently a linear combination of terms of the form $n^i (r^m)^n$.
    The only thing that has changed is the exponential terms themselves, so if
    \begin{equation*}
        \prod_{k = 1}^d (x - r_k)
    \end{equation*}
    is the characteristic polynomial of $a(n)$ with possibly repeated roots
    $r_1, \dots, r_d$, then
    \begin{equation}
        \label{polyM}
        \prod_{k = 1}^d (x - r_k^m).
    \end{equation}
    annihilates $n \mapsto a(nm)$. From the elementary theory of polynomials,
    the coefficients of \eqref{polyM} are elementary symmetric functions of the
    roots $r_k^m$. C-finite sequences are closed under multiplication and
    addition, so the coefficients of the polynomial are C-finite with respect
    to $m$.

    To obtain the degree bound, recall that the coefficient on $x^{d - i}$ in
    \eqref{polyM} equals $(-1)^i e_i(r_1^m, \dots, r_d^m)$, where $e_i(r_1^m,
    \dots, r_d^m)$ is the sum of all products of $i$ distinct $r_k^m$. Each of
    these products is of the form $\alpha^m$ for some constant $\alpha$. The
    number of such terms is an upper bound on the degree of the sequence with
    respect to $m$, and there are exactly ${d \choose i}$ of them.

    Finally, note that the coefficient on $x^{d - 1}$ is precisely the sum
    $\sum_k r_k^m$, which is annihilated by the characteristic polynomial of
    $a(n)$ itself, and the coefficient on $x^{d - d}$ is precisely the product
    $(r_1 r_2 \dots r_d)^m$.
\end{proof}

\paragraph{Example: Perrin numbers} The Perrin numbers $P(n)$ are a third-order
C-finite sequence defined by
\begin{align*}
    P(0) = 0 \quad &P(1) = 0 \quad P(2) = 2 \\
    P(n + 3) &= P(n + 1) + P(n).
\end{align*}
They are sometimes called the ``skipponaci'' numbers. They satisfy the
interesting property that $p$ divides $P(p)$ for every prime $p$. Tracing
through the above proof reveals the meta-recurrence
\begin{equation}
    \label{perrin-rec}
    P(mn) = P(m) P(m(n - 1)) + c(m) P(m(n - 2)) + P(m(n - 3)),
\end{equation}
where $c(m)$ is A078712 in the OEIS.

\paragraph{Example: General second-order} Let $a(n)$ be annihilated by $(x -
r_1) (x - r_2)$ for distinct reals $r_1$ and $r_2$. The proof of
Proposition~\ref{uniform-rec} shows that $n \mapsto a(mn)$ is annihilated by
\begin{equation*}
    (x - r_1^m)(x - r_2^m) = x^2 - (r_1^m + r_2^m) x + (r_1 r_2)^m.
\end{equation*}
In particular, if $r_1$ and $r_2$ are the golden ratio and its conjugate,
respectively, then $r_1^m + r_2^m = L_m$ is the $m$th Lucas number, and $r_1
r_2 = -1$. This recovers \eqref{fib-rec}.

\paragraph{Example: Square Fibonacci} The square Fibonacci numbers $F_n^2$ are
also C-finite. Going through the steps of the above proof and consulting the
OEIS reveals the following general identity:
\begin{equation}
    \label{square-rec}
    F_{mn}^2 = (5 F_m^2 + 3 (-1)^m) (F_{m(n - 1)}^2 - (-1)^m F_{m(n - 2)}^2) + (-1)^m F_{m(n - 3)}^2.
\end{equation}

\paragraph{Example: Tribonacci} Consider the sequence $T_n$ defined by
\begin{align*}
    T_0 &= 0 \quad T_1 = 0 \quad T_2 = 1 \\
    T_n &= T_{n - 1} + T_{n - 2} + T_{n - 3}.
\end{align*}
The family of sequences $n \mapsto T_{nm}$ satisfy the following recurrences:
\begin{align*}
    T_{n} &= T_{n - 1} + T_{n - 2} + T_{n - 3} \\
    T_{2n} &= 3 T_{2(n - 1)} + T_{2(n - 2)} + T_{2(n - 3)} \\
    T_{3n} &= 7 T_{3(n - 1)} - 5 T_{3(n - 2)} + T_{3(n - 3)} \\
    T_{4n} &= 11 T_{4(n - 1)} + 5 T_{4(n - 2)} + T_{4(n - 3)} \\
    T_{5n} &= 21 T_{5(n - 1)} + T_{5(n - 2)} + T_{5(n - 3)} \\
    T_{6n} &= 39 T_{6(n - 1)} - 11 T_{6(n - 2)} + T_{6(n - 3)}.
\end{align*}
In general,
\begin{equation*}
    T_{nm} = c_1(m) T_{(n - 1)m} + c_2(m) T_{(n - 1)m} + T_{(n - 2)m},
\end{equation*}
where
\begin{align*}
    c_1(1) &= 1 \quad c_1(2) = 3 \quad c_1(3) = 7 \\
    c_1(m) &= c_1(m - 1) + c_1(m - 2) + c_1(m - 3)
\end{align*}
and
\begin{align*}
    c_2(1) &= 1 \quad c_1(2) = 1 \quad c_1(3) = -5 \\
    c_2(m) &= -c_2(m - 1) - c_2(m - 2) + c_2(m - 3).
\end{align*}

The sequences $c_k(m)$ were found via guessing. However,
Proposition~\ref{uniform-rec} establishes that these sequences \emph{are}
C-finite, and so proving our guess requires that we check only finitely many
terms. In this case we must check no more than double the maximum degree, which
is 6 terms. We have produced just enough examples above to constitute a proof.

\section{Uniform sums}%
\label{sec:uniform_sums}

The Fibonacci numbers satisfy the famous summation identity
\begin{equation}
    \label{fib-sum}
    \sum_{k = 0}^n F_k = F_{n + 2} - 1.
\end{equation}
There are as many ways to prove this identity as there are articles devoted to
evaluating related Fibonacci sums \cite{layman, melham, frontczak}, but the
most useful method at this juncture is the following method outlined in
\cite{concrete}. The annihilating polynomial of $F_n$ can be written as
\begin{equation*}
    x^2 - x - 1 = (x - 1) x - 1.
\end{equation*}
Applying this to $F_n$ shows that $F_n = (x - 1) F_{n + 1} = F_{n + 2} - F_{n +
1}$. If we sum over $n$, then the right-hand side telescopes and we recover
\eqref{fib-sum}. In general, if $p(x)$ annihilates $a(n)$ and $p(1) \neq 0$,
then we can write $p(x) = (x - 1) q(x) + p(1)$ for some easily-computable
polynomial $q(x)$. Applying this to $a(n)$ shows that $a(n) = (x - 1) b(n)$
where $b(n) = -q(x) a(n) / p(1)$. Summing over $n$ yields
\begin{equation*}
    \sum_{0 \leq k < n} a(k) = b(n) - b(0).
\end{equation*}
From this idea, the uniform recurrences we have derived for sequences of the
form $n \mapsto a(mn)$ and $n \mapsto a(ni) a(nj)$ will help us discover
uniform summation identities.

Here is one such identity for the Perrin numbers, using \eqref{perrin-rec}.

\begin{proposition}
    The Perrin numbers $P(n)$ satisfy
    \begin{equation*}
        \sum_{0 \leq k < n} P(mn)
        =
        \frac{(P(n) - 3) (1 - P(m) - c(m)) + P(n + 1) (1 - P(m)) + P(n + 2) - 2}{P(m) + c(m)},
    \end{equation*}
    where $c(m)$ is A078712 in the OEIS.
\end{proposition}

Using \eqref{square-rec}, we can quickly rediscover the following infinite
family of sums for the square of the Fibonacci numbers.

\begin{proposition}
    \label{fib-square-sum}
    If $m$ is odd, then
    \begin{equation*}
        \label{fib-square-sum-eq}
        \sum_{0 \leq k < n} F_{mk}^2 = \frac{F_{mn} F_{m(n - 1)}}{L_m}.
    \end{equation*}
\end{proposition}

\begin{proof}
    Using \eqref{square-rec}, we obtain
    \begin{equation*}
        \sum_{0 \leq k < n} F_{mk}^2
        =
        \frac{F_{mn}^2 (7 - 10 F_m^2) + (F_{m(n + 1)}^2 - F_m^2) (4 - 5 F_m^2) + F_{m(n + 2)}^2 - F_{2m}^2}{10 F_m^2 - 8}.
    \end{equation*}
    This is far from the most economical representation. First, the numerator
    here contains $(5 F_m^2 - 4) F_m^2 - F_{2m}^2$. It is easy to check that
    \begin{equation}
        \label{vanish}
        (5 F_m^2 - 4) F_m^2 - F_{2m}^2 = -8 F_m^2 \frac{(-1)^m + 1}{2},
    \end{equation}
    so the expression on the left vanishes when $m$ is odd. We are down to
    \begin{equation*}
        \frac{F_{mn}^2 (7 - 10 F_m^2) + F_{m(n + 1)}^2 (4 - 5 F_m^2) + F_{m(n + 2)}^2}{10 F_m^2 - 8}.
    \end{equation*}
    Applying the general recurrence \eqref{fib-rec} to $F_{m(n + 2)}$ and
    simplifying the result brings us to
    \begin{equation*}
        \frac{F_{mn}^2 (8 - 10 F_m^2) + F_{m(n + 1)}^2 ((4 - 5 F_m^2) + L_m^2) + 2 F_{mn} L_m F_{m(n + 1)}}{10 F_m^2 - 8}.
    \end{equation*}
    When $m$ is odd, the identity $4 - 5 F_m^2 + L_m^2 = 0$ follows from
    dividing \eqref{vanish} by $F_m^2$ and recalling that $L_m = F_{2m} / F_m$.
    Using this and simplifying gives
    \begin{equation*}
        \frac{F_{mn} (-L_m F_{mn} + F_{m(n + 1)})}{L_m},
    \end{equation*}
    and applying the general recurrence \eqref{fib-rec} once more to $F_{m(n +
    1)}$ gives us the final answer $F_{mn} F_{m(n - 1)} / L_m$.
\end{proof}

\section{Uniform products}%
\label{sec:uniform_products}

The proof of Proposition~\ref{uniform-rec} relied on little more than the
identity $r^{mn} = (r^m)^n$ and some structural facts about C-finite sequences.
Unsurprisingly, these ideas apply to other settings. The below proposition
shows how to apply the idea to prove that sequences of the form $n \mapsto
a(ni) a(nj)$ also satisfy meta C-finite recurrences.

\begin{proposition}
    \label{product-rec}
    If $a(n)$ is C-finite of degree $d$ whose characteristic polynomial has $m$
    distinct roots, then $P_{i, j}(n) = a(ni) a(nj)$ satisfies a recurrence of
    the form
    \begin{equation*}
        P_{i, j}(n) = \sum_{k = 1}^{m(2d - m)} c_k(i, j) P_{i, j}(n - k),
    \end{equation*}
    where each $c_k(i, j)$ is C-finite with respect to $i$ and $j$ and $c_k(i,
    j) = c_k(j, i)$. The sequence $c_k(i, j)$ has order (with respect to $i$ or
    $j$) no more than ${d \choose k}$.
\end{proposition}

\begin{proof}
    Write the characteristic polynomial of $a(n)$ as $\prod_{k = 1}^m (x -
    r_k)^{d_k + 1}$ where the $r_k$ are distinct and $d_1 + d_2 + \cdots + d_m =
    d - m$. Then,
    \begin{equation*}
        a(n) = \sum_{k = 1}^m p_k(n) r_k^n,
    \end{equation*}
    where $p_k$ is a polynomial in $n$ of degree $d_k$ or less. Therefore
    \begin{equation*}
        P_{i, j}(n) = \sum_{1 \leq k, v \leq m} p_k(in) p_v(jn) (r_k^i r_v^j)^n.
    \end{equation*}
    Immediately, we see that $P_{i, j}(n)$ is annihilated by
    \begin{equation}
        \label{product-killer}
        \prod_{1 \leq k, v \leq m} (x - r_k^i r_v^j)^{d_k + d_v + 1},
    \end{equation}
    a polynomial of degree $\sum_{k, v} (d_k + d_v + 1) = m(2d - m)$. The
    coefficients of this polynomial are elementary symmetric polynomials in the
    variables $\{r_k^i r_v^j\}_{1 \leq k, v \leq d}$, and therefore C-finite
    with respect to $i$ and $j$ by the C-finite closure properties. The roots
    $r_k^i r_v^j$ are symmetric in $i$ and $j$, so the coefficient sequences
    are as well.

    The coefficient on $x^{D - k}$ is essentially the sum of all products of
    $k$ distinct elements from $\{r_k^i r_v^j\}_{1 \leq k, v \leq d}$. As a
    sequence in $i$ the $r_v^j$ factors are irrelevant: The coefficient will be
    annihilated by the characteristic polynomial for the sum of all products of
    $k$ distinct elements from $\{r_k^i\}_{1 \leq k \leq d}$. Each term of this
    latter sum is of the form $\alpha^i$ for some constant $\alpha$, and there
    are no more than ${d \choose k}$ distinct values of $\alpha$. Therefore
    $c_k(i, j)$ has order no more than ${d \choose k}$ with respect to $i$ (and
    also $j$).
\end{proof}

The previous proof can be slightly modified to produce a stronger statement.
Namely, if we split the product \eqref{product-killer} into diagonal and
off-diagonal terms, we get the following corollary.

\begin{corollary}
    \label{factor-corollary}
    Let $a(n)$ be a C-finite sequence of degree $d$ whose characteristic
    polynomial has $m$ distinct roots. Then $n \mapsto a(ni)a(nj)$ is
    annihilated by a polynomial $C_{i, j}(x)$ which factors as
    \begin{equation}
        C_{i, j}(x) = L_{i + j}(x) R_{i, j}(x),
    \end{equation}
    where $\deg L_{i + j} = 2d - m$ and $\deg R_{i, j} = (m - 1)(2d - m)$. The
    coefficients of $L_{i + j}(x)$ are C-finite sequences in $i + j$ and the
    coefficients of $R_{i, j}(x)$ are C-finite sequences which are symmetric in
    $i$ and $j$.
\end{corollary}

There is one case of this corollary worth highlighting. Now that we know these
annihilating polynomials with C-finite coefficients exist, we could find them
by computing enough examples and guessing a pattern. However, if the degrees of
$L_{i + j}(x)$ and $R_{i, j}(x)$ are the same, then it is not always clear
which factor is $L$ and which factor is $R$ in a given example. This happens
when $2d - m = (m - 1)(2d - m)$. Since $m \leq d$, the interesting solution is
$m = 2$. Thus sequences with exactly two roots in their characteristic
polynomial should be handled ``manually.'' We will show one example.

\paragraph{Example: Second-order annihilators} Let $a(n)$ be a C-finite
sequence annihilated by the quadratic $(x - r_1)(x - r_2)$ where $r_1 \neq
r_2$. Then $n \mapsto a(ni) a(nj)$ is annihilated by
\begin{equation*}
    (x^2 - \mathcal{L}(i + j) x + (r_1 r_2)^{i + j})(x^2 - (r_1 r_2)^j \mathcal{L}(i - j) x + (r_1 r_2)^{i + j})
\end{equation*}
where $\mathcal{L}(n) = r_1^n + r_2^n$. If $a(n) = F(n)$ equals the $n$th
Fibonacci number, then $\mathcal{L}(n) = L(n)$ is the $n$th Lucas number, $r_1
r_2 = -1$, and we obtain the annihilator
\begin{equation*}
    (x^2 - L(i + j) x + (-1)^{i + j})(x^2 - (-1)^j L(i - j) x + (-1)^{i + j}).
\end{equation*}

\section{Computer demo}

This article is joined by a corresponding Maple package \texttt{MetaCfinite}, obtainable from GitHub at \url{https://github.com/rwbogl/MetaCfinite}.
With \texttt{MetaCfinite}, nearly all the propositions described in this
article can be explored and checked empirically.

\paragraph{Guessing uniform recurrences} Suppose that we want to discover
\eqref{fib-rec} and the corresponding general pattern. The following Maple
commands compute the five recurrences from \eqref{fib-rec}:
\begin{verbatim}
    Fib := [[0, 1], [1, 1]:
    mSect(Fib, 1, 0); # [[0, 1], [1, 1]]
    mSect(Fib, 2, 0); # [[0, 1], [3, -1]]
    mSect(Fib, 3, 0); # [[0, 2], [4, 1]]
    mSect(Fib, 4, 0); # [[0, 3], [7, -1]]
    mSect(Fib, 5, 0); # [[0, 5], [11, 1]]
\end{verbatim}
We are trying to guess the pattern followed by $1, 3, 4, 7, 11$, and $1, -1, 1,
-1, 1$. The following command does this for us:
\begin{verbatim}
    MetaMSect(Fib, 0); # [[[1, 3], [1, 1]], [[1], [-1]]]
\end{verbatim}
This tells us that, for example, the coefficient on $F_{m(n - 1)}$ is a
sequence $L_m$ which begins $L_1 = 1$, $L_2 = 3$, and satisfies $L_m = L_{m -
1} + L_{m - 2}$. These are the Lucas numbers.

\paragraph{Uniform summation identities} The procedure \texttt{polysum(a, n, p,
x} computes an expression for $\sum_{0 \leq k < n} a(k)$ where $a(n)$ is a
C-finite sequence with characteristic polynomial $p(x)$. For example, the
following command derives the famous identity \eqref{fib-sum}:
\begin{verbatim}
    polysum(F, n, x^2 - x - 1, x); # F(n + 1) - F(1).
\end{verbatim}
This is most powerful when joined with uniform recurrences found by
\texttt{MetaMSect}. For instance, the sequence $n \mapsto F(mn)$ has
characteristic polynomial $p_m(x) = x^2 - L(m) x - (-1)^{m + 1}$. The following
commands derive a summation identity for $\sum_{0 \leq k < n} F(mk)$:
\begin{verbatim}
    polysum(Fm, n, x^2 - L(m) * x - (-1)^(m + 1), x);
         (Fm(n) - Fm(0)) (1 - L(m)) + Fm(n + 1) - Fm(1)
       - ----------------------------------------------
                                       m
                        1 - L(m) + (-1)
\end{verbatim}
That is, we have automatically derived the famous identity
\begin{equation*}
    \sum_{0 \leq k < n} F(mk) = \frac{F(mn) (1 - L(m)) + F(m(n + 1)) - F(m)}{L(m) - 1 - (-1)^m}.
\end{equation*}

\section{Conclusion}
We have used the theory of C-finite sequences to establish \emph{meta-facts}
about the recurrences C-finite sequences satisfy. Namely, we have shown that
the recurrences satisfied by $n \mapsto a(nm)$ and $n \mapsto a(ni) a(nj)$ are
uniform in a C-finite sense. This allowed us to state uniform families of
summation identities for some C-finite sequences.

The summation identities our methods derive are automatic and uniform, but we
do not claim that they are the ``best possible.'' For instance, the first
expression obtained for $\sum_{k = 0}^{n - 1} F_{mk}^2$ in
Proposition~\ref{fib-square-sum} is quite cumbersome compared to the final
answer:
\begin{equation*}
    \frac{F_{mn}^2 (7 - 10 F_m^2) + (F_{m(n + 1)}^2 - F_m^2) (4 - 5 F_m^2) + F_{m(n + 2)}^2 - F_{2m}^2}{10 F_m^2 - 8}
    =
    \frac{F_{mn} F_{m(n - 1)}}{L_m}.
\end{equation*}
It still takes some (semi-automatic) sweat to discover this reduction. Can we
automatically discover and prove such ``complex = simple'' identities? And
might this apply to more complex sums, such as $\sum_{k = 0}^{n - 1} F_{mk}^5$?
The answer is likely yes---and perhaps a C-finite simplification algorithm
already exists---but we leave this as an open problem.

Finally, the author would like to acknowledge Doron Zeilberger for bringing
these problems to his attention and providing encouragement.

\end{document}